\documentclass{amsart}
   \usepackage[left=4cm, right=4cm, top=3.5cm, bottom=3cm]{geometry}
\usepackage{amssymb,amsmath,amsthm,amsrefs,newlfont,enumerate,array,graphicx,tikz,tkz-tab,xcolor,bbold,float,subcaption}
 \usepackage[shortlabels]{enumitem}
 \numberwithin{equation}{section}
 \newtheorem{thm}{Theorem}[section]
 \newtheorem{cor}[thm]{Corollary}
 \newtheorem{lem}[thm]{Lemma}

  \newtheorem{fact}[thm]{Fact}
 \newtheorem{prop}[thm]{Proposition}
 \theoremstyle{definition}
 
 \theoremstyle{remark}
 \newtheorem{rem}[thm]{Remark}
  \newtheorem{question}[thm]{Question}

 \newtheorem*{ex}{Example}

\usepackage[colorlinks, citecolor=blue, linkcolor=red, pdfstartview=FitB]{hyperref}

\begin{document}

\title{Schauder basis with finite Blaschke products}

\author{E. Fricain}
 \address{Univ. Lille, CNRS, UMR 8524 - Laboratoire Paul Painlevé, F-59000 Lille, France}
 \email{emmanuel.fricain@univ-lille.fr}

\author{J. Mashreghi}
\address{D\'epartement de math\'ematiques et de statistique,
         Universit\'e Laval,
         Qu\'ebec, QC,
         Canada G1K 7P4}
\email{javad.mashreghi@mat.ulaval.ca}

\author{M. Nasri}
\address{Department of Mathematics and Statistics,
         University of Winnipeg,
         Winnipeg, MN,
         Canada R3B 2E9}
\email{m.nasri@uwinnipeg.ca}

\author{M.Ostermann}
 \address{Univ. Lille, CNRS, UMR 8524 - Laboratoire Paul Painlevé, F-59000 Lille, France}
 \email{maeva.ostermann@univ-lille.fr}



\begin{abstract}
We study Schauder bases for spaces of holomorphic functions in the open unit disk $\mathbb{D}$. Given a non-Blaschke sequence $(\lambda_n)_{n\geq1}$ in $\mathbb{D}$, we show that the associated sequence of finite Blaschke products $(B_n)_{n\geq1}$ forms a Schauder basis for $\mathrm{Hol}(\overline{\mathbb{D}})$ when this space is endowed with a norm inherited from a Banach space $X$ satisfying a set of natural structural assumptions. This abstract framework includes, in particular, the classical Hardy spaces $H^p$ ($1\le p\le\infty$), the weighted Bergman spaces $A_\alpha^p$ ($1\le p\le\infty$, $\alpha>-1$), and BMOA. We further prove that if the sequence $(\lambda_n)_{n\geq1}$ is contained in a compact subset of the open unit disk, then $(B_n)_{n\geq1}$ is a Schauder basis for $\mathrm{Hol}(\overline{\mathbb{D}})$ endowed with its natural topology. Moreover, in the case of Hardy spaces $H^p$ ($1<p<\infty$), we provide a complete characterization of those sequences $(\lambda_n)_{n\geq1}$ for which the corresponding sequence of finite Blaschke products forms a Schauder basis. Finally, we discuss related phenomena in the context of the disk algebra $\mathcal{A}(\mathbb D)$.
\end{abstract}

\maketitle

\section{Introduction}
Let $\mathcal{X}$ be  a complex topological linear space. We say that the sequence $(x_n)_{n \geq 0}$ in $\mathcal{X}$ is a {\em Schauder basis} for $\mathcal{X}$ if, for each $x \in \mathcal{X}$, there is a unique sequence $(a_n)_{n\geq 0}$ of complex numbers such that $x=\sum_{n=0}^\infty a_n x_n$, where the series converges in the topology of $\mathcal{X}$. For recent developments in Schauder basis, see \cites{MR4322610, MR3856231} and the references within. In this paper, we focus on the case where $\mathcal X$ is a space of analytic functions in the open unit disk $\mathbb{D}$ of the complex plane (for example, $\mathrm{Hol}(\overline{\mathbb{D}})$, the Hardy space $H^p$ or the disk algebra $\mathcal{A}(\mathbb D)$), and the sequence $(x_n)_{n\geq 0}$ is a sequence of finite Blaschke products. More precisely, given any sequence $(\lambda_n)_{n \geq 1}$ in $\mathbb{D}$, we define the finite Blaschke products $B_0=1$ and
\begin{equation}\label{D:def-Bn}
B_n(z) = \prod_{k=1}^{n} \frac{\lambda_k-z}{1-\overline{\lambda}_kz}, \qquad n \geq 1.
\end{equation}
The sequence $(\lambda_n)_{n \geq 1}$ in $\mathbb{D}$ is called a Blaschke sequence whenever
\begin{equation}\label{E:Blaschke-condition}
\sum_{n=1}^{\infty} (1-|\lambda_n|) < \infty.
\end{equation}
In this case, under suitable normalization, the modified $B_n$ converges to the infinite Blaschke product
\[
B(z) = \prod_{k=1}^{\infty} \frac{|\lambda_k|}{\lambda_k} \, \frac{\lambda_k-z}{1-\overline{\lambda}_kz}.
\]
However, if \eqref{E:Blaschke-condition} is not fulfilled, the so called non-Blaschke sequence, then
\begin{equation}\label{E:non-blaschke}
\lim_{n \to \infty} B_n(z) = 0, \qquad z \in \mathbb{D}.
\end{equation}
In fact, the convergence to zero is uniform in compact subsets of $\mathbb{D}$.

The space $\mathrm{Hol}(\overline{\mathbb{D}})$ represents the collection of functions that are analytic in a neighborhood of the closed unit disk $\overline{\mathbb{D}}$. More explicitly, $f \in  \mathrm{Hol}(\overline{\mathbb{D}})$ whenever there is an $R_0=R_0(f)>1$ such that $f$ is analytic on the disk $R_0\mathbb{D}=D(0,R_0)$ centered at the origin with radius $R_0$. It is trivial that each finite Blaschke product belongs to $ \mathrm{Hol}(\overline{\mathbb{D}})$. This space is usually endowed with its natural topology defined as the inductive limit topology of the family of Banach spaces
$\{\mathcal A(r_n\mathbb D))\}_{n\geq 1}$, where $r_n=1+\frac{1}{n}$ and where, for an open disk $D$, we write $\mathcal A(D)=Hol(D)\cap C(\overline{D})$ (equipped with the sup norm on $D$).
In concrete terms, a sequence $(f_n)_{n\geq 0}$ converges to $f$ in
$\mathrm{Hol}(\overline{\mathbb D})$ if and only if there exists a $R_0>1$ such that $f_n, f \in \mathcal A(R_0\mathbb D)$ for all sufficiently large $n$, and
$f_n \longrightarrow f$ uniformly on $R_0\mathbb D$. Note that, with this topology, $\mathrm{Hol}(\overline{\mathbb D})$ is a complete topological vector space.  However, in this note, we also consider other topologies inherited as a subspace of some Banach spaces $X$, such as $H^p$ spaces or $BMOA$. In such settings, one of our goals is to present a Schauder basis for $\mathrm{Hol}(\overline{\mathbb{D}})$ consisting of a sequence of finite Blaschke products given by \eqref{D:def-Bn}.
\medskip

The organization of the paper is as follows. In Section \ref{S:preliminaries}, we gather some facts about Hardy spaces and Toeplitz operators. A detailed description of Hardy spaces is available in \cites{Duren, MR2261424, MR2500010}, and for Toeplitz operators, see \cites{MR1634900, MR4545809}. In Section \ref{S:technical-lemma}, we prove some technical results that will be used in the proof of the main results. The most important among them are Lemma \ref{L:Estim_NormSup}, which provides a crucial estimate for the norm of elements in the range of conjugate-analytic Toeplitz operators and Proposition \ref{Prop:unicity} which gives the unicity of the decomposition $\sum a_n B_n$, when the convergence is uniform on compact subsets of the open unit disk. In Section~\ref{S:completeness}, we give a characterization of the completeness of the sequence $(B_n)_{n\geq 0}$ in $\mathcal{A}(\mathbb D)$ and $H^p$. In Section \ref{S:schauder-basis}, we prove that, given any non-Blaschke sequence $(\lambda_n)_{n\geq 1}$, the sequence of finite Blaschke products $(B_n)_{n\geq 1}$, defined as in \eqref{D:def-Bn}, is a Schauder basis for $ \mathrm{Hol}(\overline{\mathbb{D}})$, endowed with a norm inherited from a Banach space $X$ satisfying some specific properties. 
This abstract framework encompasses classical settings such as the Hardy spaces $H^p$, $1 \leq p \leq \infty$, the weighted Bergman spaces $A_\alpha^p$, $1 \leq p \leq \infty$, $\alpha > -1$, and BMOA. We also prove that if, furthermore, the non-Blaschke sequence $(\lambda_n)_{n\geq 1}$ lives in a compact subset of the open unit disk, then the sequence $(B_n)_{n\geq 1}$ forms a Schauder basis for $ \mathrm{Hol}(\overline{\mathbb{D}})$ equipped with its natural topology. In Section~\ref{sec:Hardy}, we completely characterize the sequences $(\lambda_n)_{n\geq 1}$ so that $(B_n)_{n\geq 0}$ is a Schauder basis for $H^p$, $1<p<\infty$. Finally, in Section \ref{S:Shaprness}, we discuss the case of the disk algebra $\mathcal{A}(\mathbb D)$.

\par\medskip
We are deeply grateful to the referees for their careful reading of the paper and for their suggestions which allowed us to significantly improve several of our results.

\section{Notations and some standard facts} \label{S:preliminaries}
We denote by $\mathbb{T}$ the unit circle and by $\mathrm dm$ the normalized Lebesgue measure on $\mathbb{T}$. Recall that if $1\leq  p<\infty$, then the {\em conjugate exponent} of $p$ is the number $1<q\leq \infty$ such that $1/p+1/q=1$. According to a result of F. Riesz, the Lebesgue spaces $L^p(\mathbb{T})$ and $L^q(\mathbb{T})$ are dual to each other. The duality pairing can be written as
\begin{equation}\label{E:duality-lplq}
\langle f,g \rangle = \frac{1}{2\pi}\int_{0}^{2\pi} f(e^{it}) \, \overline{g(e^{it})} \, \mathrm dt,
\end{equation}
where $f \in L^p(\mathbb{T})$ and $g \in L^q(\mathbb{T})$. It is important to note that there are other ways to define duality pairing, and each formula has its own advantages.

The Hardy space $H^p(\mathbb{T})$, $1\leq p\le\infty$, is a closed subspace of $L^p(\mathbb{T})$ consisting of elements with vanishing negatively indexed Fourier coefficients. More explicitly, each $f \in H^p(\mathbb{T})$ is an element of $L^p(\mathbb{T})$ with the Fourier series representation
\[
f(e^{i\theta}) = \sum_{n=0}^{\infty} a_ne^{in\theta}, \qquad e^{i\theta} \in \mathbb{T}.
\]
The Hardy space can be considered equally as the family of analytic functions that live on the open unit disk $\mathbb{D}$ and satisfy the growth restriction
\[
\|f\|_{p} := \sup_{0<r<1} \left( \frac{1}{2\pi} \int_{0}^{2\pi} |f(re^{it})|^p \, \mathrm dt \right)^{1/p} < \infty,~\text{for}~ 1\le p<\infty,
\]
and, for $p=\infty$,
\[\|f\|_{\infty}:=\sup_{z\in\mathbb D}|f(z)|<\infty.\]
The notation for this setting is $H^p(\mathbb{D})$. The bridge between the two concepts is made via Fatou's theorem, which ensures the existence of radial limits almost everywhere on $\mathbb{T}$ for each $f \in H^p(\mathbb{D})$, and that the resulting boundary function is in $H^p(\mathbb{T})$ whose $L^p(\mathbb T)$-norm coincides with the $H^p$-norm of $f$. Conversely, given a function $f \in H^p(\mathbb{T})$, the Poisson integral formula provides the corresponding function in $H^p(\mathbb{D})$. Hence, due to this correspondence, we simply use $H^p$ for both cases. 

The Cauchy kernels 
\begin{equation}\label{E:cauchy-kernel}
k_{\lambda}(z) := \frac{1}{1-\overline{\lambda}z}, \qquad z,\lambda \in \mathbb{D},
\end{equation}
play a special role in the theory of Hardy spaces. Each $f \in H^p$ has the representation
\[
f(\lambda) = \frac{1}{2\pi} \int_{0}^{2\pi} \frac{f(e^{it})}{1-\lambda e^{-it}} \, \mathrm dt, \qquad f \in H^p,\,\lambda\in\mathbb{D},
\]
which, considering the duality pairing \eqref{E:duality-lplq}, can be written in the more concise form
\begin{equation}\label{E:cauchy-kernel-2}
f(\lambda) = \langle f,k_\lambda\rangle, \qquad f \in H^p,\,\lambda\in\mathbb{D}.
\end{equation}
Recall also that $\|k_\lambda\|_{\infty}=\frac1{1-|\lambda|}$ and
\begin{equation}\label{E:norme-cauchy-kernel}
\frac{1}{(1-|\lambda|^2)^{1/p}}\leq \|k_\lambda\|_{q} \leq  \frac{\kappa_q}{(1-|\lambda|^2)^{1/p}},
\end{equation}
where $\kappa_q=\|P\|_{\mathcal L(L^q,H^q)}$ if $1<q<\infty$. 
See, for example, 
\cite{Hartmann1996}*{Lemma A.2.3}. 
Here, $P$ is the M. Riesz projection of $L^p(\mathbb{T})$ onto $H^p(\mathbb{T})$, defined by
\[
P\left( \sum_{n=-\infty}^{\infty} a_ne^{in\theta} \right) := \sum_{n=0}^{\infty} a_ne^{in\theta},
\]
where the sum is the Fourier series of an element of an arbitrary $f \in L^p(\mathbb{T})$. The  M. Riesz celebrated result says that $P$ is a bounded projection whenever $1<p<\infty$. It is easy to verify that
\begin{equation}\label{E:Mreisz-1}
\langle Pf,g \rangle = \langle f,Pg \rangle
\end{equation}
for all $f \in L^p(\mathbb{T})$ and all $g \in L^q(\mathbb{T})$. 
\par\medskip
Let $\varphi$ be a bounded measurable function on $\mathbb{T}$. The Toeplitz operator, with symbol $\varphi$, on the Hardy space $H^p$ is the mapping
\[
\begin{array}{cccc}
T_{\varphi}: & H^p & \longrightarrow & H^p \\
& f & \longmapsto & P(\varphi f).
\end{array}
\]
The Toeplitz operators are bounded on $H^p$, $1<p<\infty$, and fulfill the estimation
\[
\|T_{\varphi}f\|_{p} \leq \kappa_p \|\varphi\|_{L^\infty} \|f\|_{p}, \qquad f \in H^p.
\]
One of the striking properties of conjugate analytic Toeplitz operators (corresponding to a symbol $\overline{\varphi}$, where $\varphi\in H^\infty$) is the abundance of their eigenvectors, as witnessed by the identity
\begin{equation}\label{E:toeplitz-coanalytic}
T_{\overline{\varphi}} k_\lambda = \overline{\varphi(\lambda)} \, k_\lambda, \qquad \lambda \in \mathbb{D},
\end{equation}
where $k_{\lambda}$ is the Cauchy kernel \eqref{E:cauchy-kernel}. A detailed treatment of Toeplitz operators is available in \cite{MR1634900}. 

Let $\lambda \in \mathbb{D}$, and let
\begin{equation}\label{E:def-blambda}
b_{\lambda}(z) := \frac{\lambda-z}{1-\overline{\lambda}z}, \qquad z \in \mathbb{D},
\end{equation} 
It is straightforward to see that $b_{\lambda}$ and $k_{\lambda}$ are related via the linear functional equation
\begin{equation}\label{eq:sdqsdqs23}
(1-|\lambda|^2)k_{\lambda} + \overline{\lambda} b_{\lambda} = 1, \qquad \lambda \in \mathbb{D}.
\end{equation}
\smallskip
In Section \ref{sec:Hardy}, we shall study the property of $(B_n)_{n\geq 0}$ to be a Schauder basis for $H^p$ using a family $(\phi_n)_{n\geq 1}$ related to $(B_n)_{n\geq 0}$, defined by
\begin{equation}\label{eq:Takneka-basis}
\phi_n~=~\sqrt{1-|\lambda_n|^2}B_{n-1} k_{\lambda_n}, \qquad n\ge 1,
\end{equation}
where  $B_0=1$ and $B_n$ are given by \eqref{D:def-Bn} for $n\ge1$. If $(\lambda_n)_{n\geq 1}$ is a non-Blaschke sequence in $\mathbb{D}$, then it is known that $(\phi_n)_{n\geq 1}$ is a Schauder basis for $H^p$ for all $1<p<\infty$ and, more specifically,  an orthonormal basis for $H^2$, the so called  \emph{Takenaka-Malmquist} basis \cites{CoifmanPeyriere2019,QianChenTan2014}. The following straightforward computation will be used in the sequel.

\begin{lem}\label{Fact:CalculImBaseTM}
Let $N,n\ge1$. Then we have
\[
T_{\overline{B}_N}\phi_n =
\begin{cases}
\phi_n/B_N, &\text{if }~n>N,\\
0, &\text{if }~n\le N.
\end{cases}\]
\end{lem}

\begin{proof}
Suppose first that $n>N$, then \[\frac{\phi_n}{B_N}=
\begin{cases}
\sqrt{1-|\lambda_{N+1}|^2}k_{\lambda_{N+1}}, &\text{if }~n=N+1,\\
\sqrt{1-|\lambda_n|^2}k_{\lambda_n}\prod_{k=N+1}^{n-1}b_{\lambda_k}, &\text{if }~n>N+1.
\end{cases}
\]
In particular $\phi_n/B_N$ belongs to $H^\infty$ and thus 
\[T_{\overline{B}_N}\phi_n=P(\phi_n/B_N)=\phi_n/B_N, \qquad\text{for all }~n>N.\]
Suppose now that $n\le N$. By \eqref{eq:sdqsdqs23}, we have $\overline \lambda_nb_{\lambda_n}=1-(1-|\lambda_n|^2)k_{\lambda_n}$, and then
\[
\overline B_N\phi_n 
~= \frac1{\sqrt{1-|\lambda_n|^2}}\overline B_NB_{n-1}(1-\overline\lambda_nb_{\lambda_n})
~=~ \frac1{\sqrt{1-|\lambda_n|^2}}\overline{\frac{B_N}{B_n}(b_{\lambda_n}-\lambda_n)}~\text{ on}~\mathbb{T}.\]
Since $g=\frac{B_N}{B_n}(b_{\lambda_n}-\lambda_n)\in H^\infty$ and satisfies $g(0)=0$, it follows that $P\overline g=0$ and thus
\[
T_{\overline{B}_N}\phi_n
= \frac1{\sqrt{1-|\lambda_n|^2}}P\bar g=0, \qquad\text{for all }~n\le N.\qedhere\]
\end{proof}    

\section{Technical lemmas and unicity of the decomposition} \label{S:technical-lemma}
Recall the definition of the Cauchy kernel $k_{\lambda}$ in \eqref{E:cauchy-kernel} and the definition of the Blaschke factor $b_\lambda$ in \eqref{E:def-blambda}. The following lemma is a simple decomposition formula that is needed in our main result.  It corresponds to the direct sum $H^p=\mathbb C k_{\lambda}\oplus b_{\lambda}H^p$, $\lambda\in\mathbb D$.

\begin{lem} \label{L:rep-f}
Let $\lambda \in \mathbb{D}$ and let $1\leq p\leq \infty$. Then, for each $f \in H^p$,
\begin{equation}\label{E:proj-kb}
f =  (1-|\lambda|^2)f(\lambda) k_{\lambda} + b_{\lambda} T_{\overline{b}_{\lambda}}f.
\end{equation}
\end{lem}

\begin{proof}
It is easy to verify that
\[
\overline{b_{\lambda}(z)} \, k_{\lambda}(z) = \frac{-\overline{z}}{1-\lambda \overline{z}}, \qquad z \in \mathbb{T},
\]
and thus
\begin{equation}\label{E:tbarblambda}
T_{\overline{b}_{\lambda}}(k_\lambda)=0.
\end{equation}
Another straightforward, but indirect, method is to use \eqref{E:toeplitz-coanalytic} to immediately arrive at the above relation.

Given $f \in H^p$, put
\begin{equation}\label{E:tbarblambda1}
g:=f-(1-|\lambda|^2)f(\lambda)k_\lambda.
\end{equation}
Then $g \in H^p$ and $g(\lambda)=0$. Hence, for some $h \in H^p$, we must have
\begin{equation}\label{E:tbarblambda1a}
f -(1-|\lambda|^2)f(\lambda)k_\lambda= b_{\lambda}h.
\end{equation}
This is part of the F. Riesz technique for extracting the zeros of an $H^p$-function and constitutes the preliminary step for the canonical factorization theorem. Then, in light of \eqref{E:tbarblambda}, we have
\begin{eqnarray*}
h = T_{\overline{b}_{\lambda}}(b_\lambda h)
= T_{\overline{b}_{\lambda}}\left( f -(1-|\lambda|^2)f(\lambda)k_\lambda\right)
= T_{\overline{b}_{\lambda}}f -(1-|\lambda|^2)f(\lambda) T_{\overline{b}_{\lambda}}k_{\lambda} = T_{\overline{b}_{\lambda}}f.
\end{eqnarray*}
Therefore, we can rewrite \eqref{E:tbarblambda1a} as $f =  (1-|\lambda|^2)f(\lambda) k_{\lambda} + b_{\lambda} T_{\overline{b}_{\lambda}}f$.
\end{proof}
Using \eqref{eq:sdqsdqs23}, we may also write \eqref{E:proj-kb} as
\begin{equation}\label{E:proj-kb2}
f =  f(\lambda) (1- \overline{\lambda} b_{\lambda}) + b_{\lambda} T_{\overline{b}_{\lambda}}f .
\end{equation}

\begin{rem}
It should be noted that, even if the Riesz projection $P$ is not bounded on $L^\infty(\mathbb{T})$, the identity \eqref{E:proj-kb2} implies that $T_{\overline{b}_\lambda}$ is bounded from $H^\infty$ to itself. Indeed, for every $f\in H^\infty$, observe that the function
$g=f-f(\lambda)(1-\overline{\lambda}b_\lambda)$ is in $H^\infty$ and vanishes at $\lambda$. Hence $(f-f(\lambda)(1-\overline{\lambda}b_\lambda))/b_\lambda$ is also in $H^\infty$ with the same norm as $g$. In particular, we get $T_{\overline{b}_\lambda}f\in H^\infty$, and
\[
\|T_{\overline{b}_\lambda}f\|_\infty=\|f-f(\lambda)(1-\overline{\lambda}b_\lambda)\|_\infty\leq 3\|f\|_\infty.
\]
If $B_n=\displaystyle\prod_{k=1}^n b_{\lambda_k}$ is a finite Blaschke product, since $T_{\overline{B}_n}=T_{\overline{b}_{\lambda_1}}\circ\dots\circ T_{\overline{b}_{\lambda_n}}$, we immediately see that $T_{\overline{B}_n}$ is also bounded from $H^\infty$ to itself. This is not true for a general symbol $\phi$ in $H^\infty$. Even, there are symbols $\phi$ in the disk algebra such that $T_{\overline\phi}$ is not bounded on $H^\infty$ (see, for example, Th.6.6.11 in \cite{MR2215991}). However, Lemma \ref{L:Estim_NormSup} gives a positive result in this direction for the subclass $ \mathrm{Hol}(\overline{\mathbb{D}})$.
\end{rem}

For a function $f$ and $s\ge0$, we use the notation $f_s$ for $f_s(z)=f(sz)$ (when it is well-defined). With this notation, observe that, if $f\in H^p$, $g\in H^q$ and $0<\rho<1$, then
\begin{equation}\label{E:Mreisz-2}
\langle f_\rho,g \rangle = \langle f,g_\rho \rangle.
\end{equation}
An important consequence for us of this equation will be the following lemma, where for a continuous function $g$ on $R\overline{\mathbb{D}}$, we write
\[
\|g\|_{R\mathbb{D}}=\sup_{|z|\leq R}|g(z)|.
\]

\begin{lem} \label{L:Estim_NormSup}
Let $f \in  \mathrm{Hol}(\overline{\mathbb{D}})$,  let $\phi \in H^\infty$, and let $R_0>1$  be chosen so that $f$ is analytic on the disk $R_0\mathbb{D}$. Then $T_{\overline{\phi}} f \in  \mathrm{Hol}(R_0\mathbb D)$. Moreover, for every $1\leq R'<R<R_0,$ and $\rho=1/R$, we have
\begin{equation}\label{Eq:Estim_NormSup}
    \|T_{\overline{\phi}}f\|_{R'\mathbb{D}}~\le~\frac{R}{R-R'}\|\phi_{\rho}\|_\infty\|f_R\|_{1}.
\end{equation}
\end{lem}

\begin{proof}
Fix $z\in\mathbb{D}$. Then, using \eqref{E:Mreisz-2},
\begin{align*}
T_{\overline{\phi}}f(z)~&=~\langle f,\phi k_z\rangle~=~\langle (f_R)_\rho,\phi k_{z} \rangle\\&=
~\langle f_R,\phi_\rho k_{\rho z} \rangle~=~\int_{\mathbb{T}}\frac{f_R(\zeta)\overline{\phi_\rho(\zeta)}}{1-\rho z\overline{\zeta}}\,\mathrm dm(\zeta).
\end{align*}
From this integral representation it follows that $T_{\overline{\phi}}f$ has an analytic extension to the larger disk $R\mathbb D$, for every $1<R<R_0$. Hence, it follows that $T_{\overline{\phi}}f\in \mathrm{Hol}(R_0\mathbb D)$. Moreover, we have
\begin{eqnarray*}
\|T_{\overline{\phi}}f\|_{R'\mathbb{D}}
&\le& \int_{\mathbb{T}} |f_R(\zeta)| \, \|\phi_\rho\|_\infty\left\|\frac{1}{1-\rho\overline{\zeta}z}\right\|_{R'\mathbb{D}}\,\mathrm dm(\zeta)\\
&\le& \frac 1{1-\rho R'}\|\phi_\rho\|_\infty\|f_R\|_{1}\\
&=&\frac{R}{R-R'}\|\phi_\rho\|_\infty\|f_R\|_{1}.
\end{eqnarray*}
\end{proof}

\begin{rem}\label{Rk:ContCUK}
Note that, for every $f\in H^\infty$ and every $\lambda\in\mathbb D$, thanks to \eqref{E:proj-kb2}, for every $z\in\mathbb D$ with $z\neq\lambda$, we have 
\[T_{\overline b_\lambda}f(z)=\frac{f(z)-f(\lambda)(1-\overline \lambda b_\lambda(z))}{b_\lambda(z)}=\frac{f(z)-f(\lambda)}{b_\lambda(z)}+\overline\lambda f(\lambda)=(|\lambda|^2-1)\frac{f(z)-f(\lambda)}{z-\lambda}+\overline\lambda f(z).\]
Using the difference quotient mapping $Q_\lambda$, this relation can be written more compactly as  
\[T_{\overline b_\lambda}= (|\lambda|^2-1)Q_\lambda+\overline\lambda I.\]
This observation shows that $T_{\overline b_\lambda}$ can be extended into a (unique) continuous linear map from $ \mathrm{Hol}(\mathbb{D})$ to itself, where $\mathrm{Hol}(\mathbb{D})$ is equipped with the topology of uniform convergence on compact subsets of $\mathbb{D}$. Thus, $T_{\overline B_n}$ can also be continuously extended from $\mathrm{Hol}(\mathbb{D})$ to itself, whenever $B_n$ is a finite Blaschke product.
\end{rem}
As a direct consequence of this remark, we have the following result. 

\begin{prop}\label{Prop:unicity}Let $(\lambda_n)_{n\geq 1}$ be a sequence in $\mathbb{D}$, define $B_n,\,n\ge0$ as in \eqref{D:def-Bn}.
Let $f\in  \mathrm{Hol}(\mathbb{D})$ and suppose that there exists a sequence $(a_n)_{n\ge0}$ such that 
\begin{equation}\label{E:explicit-coefficients} 
f=\sum_{n=0}^\infty a_nB_n,
\end{equation}
where the series converges in $\mathrm{Hol}(\mathbb{D})$ (i.e., uniformly on every compact subset of $\mathbb{D}$). Then the coefficients $(a_n)_{n\geq 0}$ are unique and given by 
\begin{equation}\label{Eq:coefs}
    a_n~=~ (T_{\overline{B}_{n}}f)(\lambda_{n+1}) - \overline{\lambda}_n (T_{\overline{B}_{n-1}}f)(\lambda_n),\qquad n\geq 0, 
\end{equation}
where $B_{-1}=0$. 
\end{prop}

\begin{proof}
Let $f\in  \mathrm{Hol}({\mathbb{D}})$, and assume that there exists a sequence $(a_n)_{n\geq 0}$ of complex numbers such that \eqref{E:explicit-coefficients} holds. By Remark \ref{Rk:ContCUK}, we have
\[
T_{\overline B_k}f=\sum_{n=0}^\infty a_n T_{\overline B_k}B_n,
\]
where the series also converges in $\mathrm{Hol}(\mathbb{D})$, and then in particular pointwise. 
Observe that $\overline{B_k/B_n}\in\overline{H^\infty}$ for $n\leq k$, and $B_n/B_k\in H^\infty$ for $n>k$. Hence, $P(\overline{B_k/B_n}) = \overline{(B_k/B_n)(0)}$ for $n\leq k$, and $P(B_n/B_k)=B_n/B_k$ for $n>k$, which gives
\[
T_{\overline B_k}B_n=\begin{cases}
\overline{(B_k/B_n)(0)},&\text{if }n\leq k,\\
B_n/B_k, &\text{if }n>k.
\end{cases}
\]
We then deduce that
\[
T_{\overline B_k}f=\sum_{n=0}^k a_n\overline{(B_k/B_n)(0)}+\sum_{n=k+1}^{\infty}a_nB_n/B_k.
\]
Since $(B_n/B_k)(\lambda_{k+1})=0$ for $n\geq k+1$, we get
\[
(T_{\overline B_k}f)(\lambda_{k+1})=\sum_{n=0}^k a_n\overline{(B_k/B_n)(0)}
\]
Use this equality to write
\begin{eqnarray*}
&& (T_{\overline B_k}f)(\lambda_{k+1})-\overline\lambda_k (T_{\overline B_{k-1}}f)(\lambda_{k})\\
&=& \sum_{n=0}^k a_n\overline{(B_k/B_n)(0)} - \overline\lambda_k\sum_{n=0}^{k-1} a_n\overline{(B_{k-1}/B_n)(0)}\\
&=&\sum_{n=0}^{k-1}a_n\left(\overline{(B_k/B_n)(0)}-\overline\lambda_k\overline{(B_{k-1}/B_n)(0)}\right) + a_k\overline{(B_k/B_k)(0)}.
\end{eqnarray*}
Observe now that $\overline{B_k/B_k}=1$ and $\overline {B_k/B_n}=\overline b_{\lambda_k}\overline {B_{k-1}/B_n}$, whence
\[
\overline{(B_k/B_n)(0)}-\overline\lambda_k\overline{(B_{k-1}/B_n)(0)}=\overline{b_{\lambda_k}(0)}\overline{(B_{k-1}/B_n)(0)}-\overline\lambda_k\overline{(B_{k-1}/B_n)(0)}=0,
\]
since $b_{\lambda_k}(0)=\lambda_k$. Hence, $a_k=(T_{\overline B_k}f)(\lambda_{k+1})-\overline\lambda_k (T_{\overline B_{k-1}}f)(\lambda_{k})$, which reveals the uniqueness of the coefficients $a_n$, $n\geq 0$, and this concludes the proof.
\end{proof}
\begin{rem}
If the sequence $(\lambda_n)_{n\geq 1}$ is a sequence in $\mathbb{D}$ of distinct points, then we have a simpler proof of the unicity which only uses the pointwise convergence. Indeed, let $f\in  \mathrm{Hol}(\mathbb{D})$ and assume that there exists a sequence $(a_n)_{n\geq 0}$ of complex numbers such that we have the representation \eqref{E:explicit-coefficients}, where the series is pointwise convergent. In particular, we can evaluate the identity  \eqref{E:explicit-coefficients} at $\lambda_k$, $k\geq 1$. Hence,
\[
f(\lambda_k)=\sum_{n=0}^{k-1}a_n B_n(\lambda_k),
\]
since $B_n(\lambda_k)=0$ for $n\geq k$. Thus, the sequence $(a_0,a_1,\dots,a_{k-1})$ is a solution of the linear system $A a=b$, where
\[
a=\begin{pmatrix}a_0\\a_1\\\vdots\\a_{k-1}
\end{pmatrix},
\,\,
b=\begin{pmatrix}f(\lambda_1)\\f(\lambda_2)\\\vdots \\ f(\lambda_k)
\end{pmatrix},
\,\, \text{and} \,\,
A=\begin{pmatrix}
1&0&\dots&0\\
1&B_1(\lambda_2)&\dots&0\\
\vdots& \vdots &\ \ddots& \vdots\\
1&B_1(\lambda_k)&\dots&B_{k-1}(\lambda_k)
\end{pmatrix}.
\]
Since $A$ is lower triangular with non-zero elements on its diagonal, $A$ is invertible and thus the solution of the system is unique.
\end{rem}

\section{Completeness of the sequence of finite Blaschke products}\label{S:completeness}

Before studying the property of being a Schauder basis, it is natural to first study the completeness. The following result gives the characterization of the completeness of  the family $(B_n)_{n\geq 0}$ in $\mathcal{A}(\mathbb{D})$ and $H^p$.
\begin{thm}\label{prop:completude-disc-algebra}
Let $(\lambda_n)_{n\geq 1}$ be a sequence of points in $\mathbb{D}$, define $B_n,\,n\ge0$, as in \eqref{D:def-Bn}. Let $1\leq p<\infty$. Then the following assertions are equivalent.
\begin{enumerate}
\item[(i)] The sequence $(B_n)_{n\geq 0}$ is complete in $\mathcal{A}(\mathbb{D})$.
\item[(ii)] The sequence $(B_n)_{n\geq 0}$ is complete in $H^p$.
\item[(iii)] The sequence $(\lambda_n)_{n\geq 1}$ is non-Blaschke, i.e., $\sum_{n=1}^{\infty}(1-|\lambda_n|)=\infty$.
\end{enumerate}
In particular, if $(\lambda_n)_{n\geq 1}$ is a 
Blaschke sequence, then the sequence $(B_n)_{n\geq 0}$ cannot be complete in $\mathrm{Hol}(\overline{\mathbb{D}})$ equipped with the topology induced by $\mathcal A(\mathbb D)$ or $H^p$.
\end{thm}

\begin{proof}
$(i)\implies (iii):$ Argue by contradiction and assume that  $(B_n)_{n\geq 1}$ is complete in $\mathcal{A}(\mathbb{D})$ and  $\sum_{n=1}^{\infty}(1-|\lambda_n|)<\infty$. Then we can consider the infinite Blaschke product $B$ associated with $(\lambda_n)_{n\geq 1}$. Consider the linear form 
\[
\varphi(f)=\int_{\mathbb{T}}f(\zeta)\overline{\zeta B(\zeta)}\,\mathrm dm(\zeta),\qquad f\in \mathcal{A}(\mathbb{D}).
\]
It is easy to check that $\varphi$ is an element of the dual of $\mathcal{A}(\mathbb{D})$ that is not identically zero (for instance if $\lambda\neq \lambda_n$, $n\geq 1$, $\varphi(zk_\lambda)=\overline{B(\lambda)}\neq 0$). On the other hand, for every $n\geq 1$, we have 
\[
\varphi(B_n)=\int_{\mathbb{T}}B_n(\zeta)\overline{\zeta B(\zeta)}\,\mathrm dm(\zeta)=\overline{\left(z\frac{B}{B_n}\right)(0)}=0.
\]
This contradicts the completeness of $(B_n)_{n\geq 1}$.
\par\smallskip
$(iii)\implies (i):$ Assume that $\sum_{n=1}^{\infty}(1-|\lambda_n|)=\infty$. Let $\varphi$ be in the dual of $\mathcal{A}(\mathbb{D})$ and assume that $\varphi(B_n)=0$ for every $n\geq 1$. Using the Hahn-Banach theorem, we can extend $\varphi$ into a continuous linear form on $\mathcal{C}(\mathbb{T})$, and then by the Riesz representation theorem there exists a complex measure $\mu$ on $\mathbb{T}$ such that 
\[
\varphi(f)=\int_{\mathbb{T}}f(\zeta)\,d\mu(\zeta),\qquad f\in \mathcal{A}(\mathbb{D}).
\]
\begin{fact}\label{fact-vect}
We have $\mathrm{vect}(k_{\lambda_n}:n\geq 1)\subset \mathrm{vect}(B_n:n\geq 0)$.
\end{fact}
\begin{proof}Since $\sum_{n=1}^\infty (1-|\lambda_n|)=\infty$, the sequence $(\phi_n)_{n\geq 1}$ given by \eqref{eq:Takneka-basis} is an orthonormal basis of $H^2$. Using \eqref{eq:sdqsdqs23}, we have 
\[
\phi_n=(1-|\lambda_n|^2)^{-1/2}(B_{n-1}-\overline\lambda_n B_n),
\]
 and in particular, for every $n\geq 0$, we get
\[
\mathrm{vect}(\phi_1,\dots,\phi_n)\subset \mathrm{vect}(B_0,B_1,\dots,B_n).
\]
Assume now $\lambda\in\{\lambda_n:n\geq 1\}$. Then there exists $k\geq 2$ such that $\phi_{k-1}(\lambda)\neq 0$ and $\phi_k(\lambda)=0$. This implies that for every $n\geq k$, $\phi_n(\lambda)=0$. Hence
\[
k_\lambda\in \left(\mathrm{vect}(\phi_n:n\geq k)\right)^\perp=\mathrm{vect}(\phi_1,\dots,\phi_{k-1}),
\]
where here the symbol $\perp$ means the orthogonal complement in $H^2$. Finally, we get $k_\lambda\in\mathrm{vect}(B_n:n\geq 0)$, which proves the Fact \ref{fact-vect}.
\end{proof}
Let us mention that the proof of Fact \ref{fact-vect} does not really use the property that $(\lambda_n)_{n\geq 1}$ is a non-Blaschke sequence: indeed, for a fixed $l\ge1$, the inclusion $\mathrm{vect}(k_{\lambda_1},\dots,k_{\lambda_l})\subset \mathrm{vect}(B_0,\dots,B_l)$ depends only on the first terms and if necessary, we can consider a non-Blaschke sequence $(\widetilde \lambda_n)$ such that $\widetilde \lambda_j=\lambda_j$ for all $j\le l$ to obtain this inclusion.
\par\medskip
Let us now return to the proof of Proposition \ref{prop:completude-disc-algebra}. Since $\varphi(B_n)=0$ for every $n\geq 0$, we get $\varphi(k_{\lambda_n})=0$ for every $n\geq 1$. This means that 
\[
\int_{\mathbb{T}}\frac{1}{1-\overline{\lambda_n}\zeta}\,\mathrm d\mu(\zeta)=0.
\]
If we denote by $C_{\overline \mu}$ the Cauchy transform of $\overline{\mu}$, the last identity means that $C_{\overline\mu}(\lambda_n)=0$ for every $n\geq 1$. But $C_{\overline \mu}\in H^p$ for every $0<p<1$ (see \cite{MR2215991}*{Theorem 2.1.10}). Hence, since $(\lambda_n)_{n\geq 1}$ is a non-Blaschke sequence, we deduce $C_{\overline \mu}=0$, and thus $\mu=0$. This proves the completeness of $(B_n)_{n\geq 0}$.

The equivalence between $(ii)$ and $(iii)$ follows the same lines.
\end{proof}

\section{The Schauder basis for $\mathrm{Hol}(\overline{\mathbb{D}})$} \label{S:schauder-basis}
According to Theorem~\ref{prop:completude-disc-algebra}, in our study of Schauder bases, we shall consider sequences $(\lambda_n)_{n\geq 1}$ which do not satisfy the Blaschke condition.

We are ready to present a sequence of finite Blaschke products which forms a Schauder basis for $\mathrm{Hol}(\overline{\mathbb{D}})$ equipped with the $H^\infty$-norm.

\begin{thm}\label{TheoX}
Let $(\lambda_n)_{n\geq 1}$ be a non-Blaschke sequence in $\mathbb{D}$, define $B_n,\,n\ge0$ as in \eqref{D:def-Bn}. Then, for each $f\in  \mathrm{Hol}(\overline{\mathbb{D}})$, there exists a unique sequence of scalars $(a_n)_{n\geq 0}$ such that we have
\begin{equation}\label{E:explicit-coefficients2} 
f=\sum_{n=0}^\infty a_nB_n,
\end{equation}
where the series converges uniformly on $\overline{\mathbb D}$ and the coefficients $a_n$ are given by \eqref{Eq:coefs}.

In particular, the sequence $(B_n)_{n\geq 0}$ is a Schauder basis for $ \mathrm{Hol}(\overline{\mathbb{D}})$ equipped with the $H^\infty$-norm.
\end{thm}

\begin{proof}
Since the uniform convergence on $\overline{\mathbb D}$ implies the convergence in $ \mathrm{Hol}(\mathbb{D})$, the unicity of the coefficients $(a_n)_{n\geq 0}$ is given by Proposition \ref{Prop:unicity}.
Let us prove the existence of the representation \eqref{E:explicit-coefficients2}. For this purpose, we shall prove that the series $\sum a_n B_n$ converges to $f$ uniformly on $\overline{\mathbb D}$ where the coefficients $a_n$ are given by \eqref{Eq:coefs}. Repeated application of the functional equation \eqref{E:proj-kb2}, with different values for $\lambda$ and $f$ in each step, gives
\begin{eqnarray*}
f &=&   T_{\overline{B}_{0}}f(\lambda_1) \, (1- \overline{\lambda}_1 b_{\lambda_1}) + b_{\lambda_1} T_{\overline{B}_{1}}f,\\
T_{\overline{B}_{1}}f &=&  T_{\overline{B}_{1}}f(\lambda_2) \, (1- \overline{\lambda}_2 b_{\lambda_2}) + b_{\lambda_2} T_{\overline{B}_{2}}f,\\
&\vdots&\\
T_{\overline{B}_{N-1}}f &=&  T_{\overline{B}_{N-1}}f(\lambda_{N}) \, (1- \overline{\lambda}_N b_{\lambda_N}) + b_{\lambda_{N}} T_{\overline{B}_{N}}f.
\end{eqnarray*}
Hence, noting that $B_{n-1}b_{\lambda_n} = B_{n}$, after some telescoping eliminations, we obtain
\begin{equation}\label{E:our-B-rep-0}
f = \sum_{n=1}^{N} (T_{\overline{B}_{n-1}}f)(\lambda_{n}) \, \left( B_{n-1} - \overline{\lambda}_n B_n\right) + B_{N} T_{\overline{B}_{N}}f.
\end{equation}
Then, rearranging the terms leads to
\begin{equation}\label{E:our-B-rep}
f = \sum_{n=0}^{N-1} \bigg( (T_{\overline{B}_{n}}f)(\lambda_{n+1}) - \overline{\lambda}_n (T_{\overline{B}_{n-1}}f)(\lambda_n) \bigg) B_n + R_Nf,
\end{equation}
where $B_{-1}=0$ and the remainder is
\begin{equation}\label{E:remainder}
R_Nf := \big( - \overline{\lambda}_N (T_{\overline{B}_{N-1}}f)(\lambda_N) +  T_{\overline{B}_{N}}f \big) B_{N}.
\end{equation}
We need to show that, for any fixed $f \in  \mathrm{Hol}(\overline{\mathbb{D}})$,
\begin{equation}\label{E:remainder-to0}
\lim_{N \to \infty} \|R_Nf\|_{\infty} = 0,
\end{equation}
which implies, with \eqref{E:our-B-rep}, the validity of the series representation \eqref{E:explicit-coefficients2}, with convergence in $H^\infty$. 

Let $f \in  \mathrm{Hol}(\overline{\mathbb{D}})$ and let $R_0>1$ be a dilation factor such that $f_{R_0} \in   \mathrm{Hol}(\overline{\mathbb{D}})$. Then, according to Lemma \ref{L:Estim_NormSup} (with $R'=1)$ and \eqref{E:remainder}, for every $1<R<R_0$, we have
\begin{equation}\label{Eq:BorneUnif}
\|R_Nf\|_\infty ~\le~\left|T_{\overline{B}_{N-1}}f(\lambda_N)\right|+\|T_{\overline{B}_N}f\|_\infty\le~\frac{2R}{R-1}\|f_R\|_{1}\sup_{\frac1R\mathbb{D}}|B_{N-1}|.
\end{equation}
But since $(\lambda_n)_{n\geq 1}$ is a non-Blaschke sequence, it follows that $B_n\longrightarrow0$ uniformly on every compact subset of $\mathbb{D}$. In particular $\sup_{\frac1R\mathbb{D}}|B_{N-1}|\longrightarrow0$ as $N\to \infty$, and thus 
\[\lim_{N \to \infty} \|R_Nf\|_{\infty} = 0.\qedhere\]
\end{proof}

Thus, we immediately obtain the following.
\begin{cor}\label{Cor:hp-blaschke}
Let $X$ be a Banach space of analytic functions on $\mathbb{D}$ that satisfies 
\[
\mathcal{A}(\mathbb{D}) \lhook\joinrel\longrightarrow X
\lhook\joinrel\longrightarrow  \mathrm{Hol}(\mathbb{D}). 
\]
More explicitly, 
\begin{enumerate}[(i)]
\item\label{Hyp2} the disk algebra $\mathcal{A}(\mathbb{D})$ is continuously embedded in $X$,
\item\label{Hyp1} and $X$ is continuously embedded in $ \mathrm{Hol}(\mathbb{D})$.
\end{enumerate}
Let $(\lambda_n)_{n\geq 1}$ be a non-Blaschke sequence in $\mathbb{D}$, and define $B_n,\,n\ge0$ as in \eqref{D:def-Bn}.
Then the sequence of finite Blaschke products $(B_n)_{n\ge0}$ forms a Schauder basis of $ \mathrm{Hol}(\overline{\mathbb{D}})$ equipped with the norm of $X$.
\end{cor}
\begin{proof}
It follows from \ref{Hyp2} that the uniform convergence in $\mathbb D$ implies the convergence in $X$, and thus the existence of the decomposition is given by Theorem \ref{TheoX}. The unicity of the coefficients in the decomposition follows from Proposition \ref{Prop:unicity}.
\end{proof}

\begin{ex}
Most of the relevant classical function spaces satisfy the hypothesis of Corollary~\ref{Cor:hp-blaschke}, such as the Hardy spaces $H^p$, the space BMOA, the weighted Bergman spaces $A_\alpha^p$, $1\leq p\leq \infty$, $\alpha>-1$, and of course the disk algebra $\mathcal{A}(\mathbb{D})$.
\end{ex}

As far as the property of being a Schauder basis for $ \mathrm{Hol}(\overline{\mathbb{D}})$ equipped with its natural topology, we have a complete characterization. Recall that convergence in this topology corresponds to uniform convergence in a neighborhood of the closed unit disk, or equivalently in a disk of the form $R'\mathbb D$ for some $R'>1$.

\begin{thm}\label{thm:base-de-schauder-HolDbar}
Let $(\lambda_n)_{n\geq 1}$ be a sequence of points in $\mathbb D$, and let $(B_n)_{n\geq 0}$ be given by \eqref{D:def-Bn}. Then the following assertions are equivalent.
\begin{enumerate}
\item[(i)] The sequence $(B_n)_{n\geq 1}$ is a Schauder basis for $ \mathrm{Hol}(\overline{\mathbb{D}})$ equipped with its natural topology. 
\item[(ii)] We have $\sup_{n\geq 1}|\lambda_n|<1$.
\end{enumerate}
\end{thm}

\emph{Proof.}$\,\,$ $(i)\implies (ii)$: If $\sup_{n\geq 1}|\lambda_n| = 1$, then the series $\sum_{n} a_n B_n$ converges in the natural topology of $ \mathrm{Hol}(\overline{\mathbb{D}})$ if and only if the sequence $(a_n)_{n\geq 1}$ is of finite support and, consequently, in this case, $(B_n)_{n\geq 0}$ cannot be a Schauder basis.

\par\medskip
$(ii)\implies (i)$: Assume now that $\sup_{n\geq 1}|\lambda_n|<1$. The proof starts by following the same arguments as in the proof of Theorem~\ref{TheoX}.  We need to show that, for each $f\in  \mathrm{Hol}(\overline{\mathbb{D}})$, there exist  $R'>1$ and a unique sequence of scalars $(a_n)_{n\geq 0}$ such that we have
\[f=\sum_{n=0}^\infty a_nB_n,\]
where the series converges uniformly on $R'\mathbb{D}$. The unicity of the decomposition is given by Proposition \ref{Prop:unicity}, and for the existence, we will show that 
$R_N f$ converges uniformly to $0$ in $R'\mathbb D$. However, this time, we will have $\|B_N\|_{R'\mathbb{D}}$ in the estimate obtained, and this term is generally unbounded. The following lemma will allow us to overcome this difficulty. 



\begin{lem}\label{lem:sdsdqsd232E}
Let $R_0>1$. Fix any $1<R<R_0$ and put $\rho=1/R$. Then there exist $1<R'<R$ and $\delta\in (0,1)$ such that, for every $\lambda$ satisfying $|\lambda|\leq 1/R_0$, we have
\[
\|b_\lambda\|_{\rho \mathbb{D}} \|b_\lambda\|_{R' \mathbb{D}}\leq \delta.
\]
\end{lem}

\begin{proof}
For every $0<r<1/|\lambda|$ and every $\zeta\in\mathbb{T}$, we have 
\[
|b_\lambda(r\zeta)|^2=1-\frac{(1-r^2)(1-|\lambda|^2)}{|1-r\overline\lambda \zeta|^2}.
\]
If $0<r\leq 1$, it is easy to see that 
\[
\sup_{\zeta\in\mathbb{T}}|b_\lambda(r\zeta)|^2=1-\frac{(1-r^2)(1-|\lambda|^2)}{(1+|\lambda|r)^2}=\frac{(|\lambda|+r)^2}{(1+|\lambda|r)^2},
\]
and if $1<r<1/|\lambda|$, we have 
\[
\sup_{\zeta\in\mathbb{T}}|b_\lambda(r\zeta)|^2=1-\frac{(1-r^2)(1-|\lambda|^2)}{(1-|\lambda|r)^2}=\frac{(r-|\lambda|)^2}{(1-|\lambda|r)^2}.
\]
Hence, according to the maximum principle, we get 
\begin{equation}\label{eq:sdqsdq23ZEE0SDMMMM}
\|b_\lambda\|_{r\mathbb{D}}=\begin{cases}
\displaystyle\frac{|\lambda|+r}{1+|\lambda|r}, &\text{if }0<r\leq 1,\\
\displaystyle\frac{r-|\lambda|}{1-|\lambda|r}, & \text{if }1<r<\frac{1}{|\lambda|}.
\end{cases}
\end{equation}
 Then it follows that, for all $1<R'<R<R_0$, $\rho=1/R$, and $|\lambda|\leq 1/R_0$, we have 
\[
\|b_\lambda\|_{\rho\mathbb{D}}\|b_\lambda\|_{R'\mathbb{D}}=\frac{|\lambda|+\rho}{1+|\lambda|\rho}\cdot \frac{R'-|\lambda|}{1-|\lambda|R'}=\frac{1+R|\lambda|}{R+|\lambda|}\cdot \frac{R'-|\lambda|}{1-|\lambda|R'}.
\]
Thus,
\begin{align*}
1 - \|b_\lambda\|_{\rho\mathbb{D}}\|b_\lambda\|_{R'\mathbb{D}}
=&
\frac{(R-R')|\lambda|^2-2(RR'-1)|\lambda|+(R-R')}{(R+|\lambda|)(1-|\lambda|R')}\\
\geq & \frac{(R-R')|\lambda|^2-2(RR'-1)|\lambda|+(R-R')}{2R_0}.
\end{align*}
The quadratic equation $(R-R')t^2-2(RR'-1)t+(R-R')=0$ has two positive zeros $0<t_1<1$ and $t_2 = 1/t_1>1$. More explicitly,
\[
t_1 = \frac{R-R'}{RR'-1 + \sqrt{(R^2-1)(R'^2-1)}},
\]
which shows $t_1 \to 1^{-}$ as $R' \to 1^+$. Choose $R'$ close enough to $1^+$ so that $1/R_0 < t_1$. Hence, for every $\lambda$ satisfying $|\lambda|\leq 1/R_0$, we have
\begin{align*}
1 - \|b_\lambda\|_{\rho\mathbb{D}}\|b_\lambda\|_{R'\mathbb{D}}
&\geq
\frac{(R-R')/R_0^2-2(RR'-1)/R_0+(R-R')}{2R_0}\\
&= \frac{(R-R')(1+R_0^2)-2R_0(RR'-1)}{2R_0^3}\\
&\geq \frac{(R-R')(R_0-1)^2}{2R_0^3}.
\end{align*}
This concludes the proof of the lemma with $\delta=1-\frac{(R-R')(R_0-1)^2}{2R_0^3}.$
\end{proof}

\noindent
Let us now come back to the proof of the implication $(ii)\implies (i)$ of Theorem~\ref{thm:base-de-schauder-HolDbar}. As already indicated, the unicity of the coefficients $(a_n)_{n\geq 0}$ is given by Proposition \ref{Prop:unicity}. Now, since $f\in  \mathrm{Hol}(\overline{\mathbb{D}})$ and $\sup_{n\geq 1}|\lambda_n|<1$, we can find $R_0>1$ such that $\sup_{n\geq 1}|\lambda_n|<1/R_0$ and $f\in  \mathrm{Hol}(R_0 \mathbb{D})$. Take $1<R'<R<R_0$ and $\delta$ as given by Lemma~\ref{lem:sdsdqsd232E}. Then, according to \eqref{E:our-B-rep}, it is sufficient to prove that 
\[
\lim_{N\to \infty}\|R_Nf\|_{R'\mathbb{D}}=0,
\]
where $R_Nf$ is given by \eqref{E:remainder}.  But it follows from Lemma \ref{L:Estim_NormSup} that 
\[
\|R_Nf\|_{R'\mathbb{D}}\leq \frac{R}{R-R'}\|f_R\|_{1}\left(\|B_{N-1}\|_{\rho\mathbb{D}}+\|B_{N}\|_{\rho\mathbb{D}}\right)\|B_N\|_{R'\mathbb{D}},
\]
where $\rho=1/R$. Using \eqref{eq:sdqsdq23ZEE0SDMMMM} and Lemma~\ref{lem:sdsdqsd232E}, we get
\begin{align*}
\|R_Nf\|_{R'\mathbb{D}}& \leq \frac{R}{R-R'}\|f_R\|_{1}\left(\delta^{N-1}\frac{R'-|\lambda_N|}{1-|\lambda_N|R'}+\delta^N\right)\\
&\leq \frac{R}{R-R'}\|f_R\|_{1}\left(\frac{R_0^2}{R_0-R'}\delta^{N-1}+\delta^N\right).
\end{align*}
Since $0<\delta<1$, we deduce that $\|R_Nf\|_{R'\mathbb{D}}\longrightarrow 0$ as $N\to\infty$, which concludes the proof.

 \hfill$\square$

\section{Characterization for the Hardy spaces}\label{sec:Hardy}
In Corollary \ref{Cor:hp-blaschke}, we have seen that if $(\lambda_n)_{n\geq 1}$ is a non-Blaschke sequence in $\mathbb{D}$, then the sequence $(B_n)_{n\geq 0}$ given by \eqref{D:def-Bn} is a Schauder basis for $ \mathrm{Hol}(\overline{\mathbb{D}})$ equipped with the $H^p$-norm. It is natural to ask if we can replace $ \mathrm{Hol}(\overline{\mathbb{D}})$ with $H^p$. The following result answers this question.

\begin{thm}\label{Th:CaracHp}
Let $1<p<\infty$, let $(\lambda_n)_{n\geq 1}$ be a  sequence in $\mathbb{D}$ and let $(B_n)_{n\geq 0}$ be given by \eqref{D:def-Bn}. Then $(B_n)_{n\geq 1}$ is a Schauder basis for $H^p$ if and only if $\sup_{n\geq 1}|\lambda_n|<1$.
\end{thm}
The key for the characterization given by Theorem \ref{Th:CaracHp}, is to relate the rest $R_Nf$ with the coefficients of $f$ in its development in the Takenaka-Malmquist basis $(\phi_n)_{n\geq 0}$ given by \eqref{eq:Takneka-basis}. 
\begin{proof}[Proof of Theorem \ref{Th:CaracHp}]
According to Theorem~\ref{prop:completude-disc-algebra}
 we may assume that $(\lambda_n)_{n\geq 1}$ is a non-Blaschke sequence. Since the convergence in $H^p$ implies the uniform convergence on every compact subset of $\mathbb{D}$, if $f=\sum_{n\geq 0} a_n B_n$ with the convergence in $H^p$, then, according to Proposition \ref{Prop:unicity}, the coefficients $a_n$ are given by 
\[
a_n~=~ (T_{\overline{B}_{n}}f)(\lambda_{n+1}) - \overline{\lambda}_n (T_{\overline{B}_{n-1}}f)(\lambda_n),\quad n\geq 0.
\]
Hence it follows from \eqref{E:our-B-rep} and \eqref{E:remainder} that $(B_n)_{n\geq 0}$ is a Schauder basis for $H^p$ if and only if for every  $f\in H^p$, $\|R_Nf\|_{p}\longrightarrow 0$ when $N\to\infty$, where
\begin{equation}\label{eq:sdqsds334ZEDDS}
R_Nf := - \overline{\lambda}_N (T_{\overline{B}_{N-1}}f)(\lambda_N)B_N +  B_NT_{\overline{B}_{N}}f.
\end{equation}
\par\smallskip
Let $f\in H^p$. Since $(\lambda_n)_{n\geq 1}$ is a non-Blaschke sequence, there exists a unique sequence $(c_n)_{n\ge1}$ such that 
\begin{equation}\label{Eq:DevfTM}
    f=\sum_{n=1}^\infty c_n\phi_n,
\end{equation}
where the series converges in the $H^p$-norm. Since for every $N\ge1$, the operators $T_{\overline{B}_N}$ are bounded on $H^p$, it follows that 
\[
T_{\overline{B}_N}f=\sum_{n=1}^\infty c_nT_{\overline{B}_N}\phi_n,\]
where the series is convergent in $H^p$. By Lemma \ref{Fact:CalculImBaseTM}, we get
\[B_NT_{\overline{B}_N}f=\sum_{n=1}^\infty c_nB_NT_{\overline{B}_N}\phi_n=\sum_{n=N+1}^\infty c_n\phi_n.\]
Since the series in \eqref{Eq:DevfTM} converges in the $H^p$-norm, it follows that 
\begin{equation}\label{eq:sdsqSdsq12dDF90sq}
\|B_NT_{\overline{B}_N}f\|_{p}=\left\|\sum_{n=N+1}^\infty c_n\phi_n\right\|_{p}\longrightarrow0~\qquad\text{when}~N\to\infty.
\end{equation}
Moreover, observe that 
\[(T_{\overline{B}_{N-1}}f)(\lambda_N)=\frac{c_N}{\sqrt{1-|\lambda_N|^2}}.
\]
Then it follows from \eqref{eq:sdqsds334ZEDDS} and \eqref{eq:sdsqSdsq12dDF90sq} that $(B_n)_{n\geq 0}$ is a Schauder basis for $H^p$ if and only if, for every $f\in H^p$, we have
\begin{equation}\label{eq:sdsd3EZ3SD099ER}
\frac{|\lambda_N c_N|}{\sqrt{1-|\lambda_N|^2}}\longrightarrow 0,\quad \text{as }N\to\infty,   
\end{equation}
where $(c_N)_{N\geq 0}$ are the coefficients in the decomposition \eqref{Eq:DevfTM} of $f$.
\par\medskip
Suppose first that $\sup_{n\geq 1}|\lambda_n|=r<1$. Let $f\in H^p$  and let $(c_n)_{n\ge1}$ as in \eqref{Eq:DevfTM}. Since the series converges in $H^p$, we have $\|c_N\phi_N\|_{p}\longrightarrow0$ when $N\to\infty$. Moreover, by \eqref{E:norme-cauchy-kernel}, we easily see that 
\[(1-|\lambda_N|^2)^{\frac12-\frac1q}\le \|\phi_N\|_{p}\le \kappa_p(1-|\lambda_N|^2)^{\frac12-\frac1q},\]
where $1/p+1/q=1$. In particular, since $\sup_{n\geq 1}|\lambda_n|<1$, we have $\inf_N\|\phi_N\|_{p}>0$, and thus $c_N\longrightarrow 0$ as $N\to\infty$. Then, it follows
\[
\frac{|\lambda_N c_N|}{\sqrt{1-|\lambda_N|^2}}\le \frac{|c_N|}{\sqrt{1-r^2}}\longrightarrow 0\quad\text{as }N\to \infty.\]
So \eqref{eq:sdsd3EZ3SD099ER} is satisfied and thus $(B_n)_{n\geq 0}$ is a Schauder basis for $H^p$.
\par\medskip
Suppose now that $\sup_{n\geq 1}|\lambda_n|=1$. Then it is sufficient to prove that there exists a lacunary sequence $(c_n)_{n\geq 1}$ such that  
\begin{equation}\label{eq:sdqssqd3EZESDSXVF00SDN}
\sum_{n=1}^\infty |c_n|\|\phi_n\|_{p}<\infty\quad\text{and}\quad\limsup_{N\to\infty}\frac{|\lambda_N c_N|}{\sqrt{1-|\lambda_N|^2}}=+\infty.
\end{equation}
Indeed, for such a choice of $(c_n)_{n\geq 1}$, the series  $f=\sum_{n=1}^\infty c_n\phi_n$ converges in $H^p$ and \eqref{eq:sdsd3EZ3SD099ER} is not satisfied for this function $f$, which proves that $(B_n)_{n\geq 0}$ is not a Schauder basis for $H^p$.
\par\smallskip
Let us now construct the sequence $(c_n)_{n\geq 1}$ satisfying \eqref{eq:sdqssqd3EZESDSXVF00SDN}.
Since $1/q-1<0$ and $\sup_{n\geq 1}|\lambda_n|=1$, it follows that $\limsup_{n\to\infty}|\lambda_n|(1-|\lambda_n|^2)^{1/q-1}=+\infty$. So there exists a subsequence $(\lambda_{n_k})_{k\geq 1}$ of $(\lambda_n)_{n\geq 1}$ such that 
\[|\lambda_{n_k}|(1-|\lambda_{n_k}|^2)^{1/q-1}\ge k^3,\qquad k\ge1.
\]
Now let $\alpha_{n_k}= 1/k^2$ for all $k\ge1$ and $\alpha_n=0$ if $n\neq n_k$ for any $k$. Then  the series $\sum_n \alpha_n$ converges. Write $c_n=\alpha_n (1-|\lambda_n|^2)^{1/q-1/2}$. Since  $\|\phi_n\|_{p}\le \kappa_p(1-|\lambda_n|^2)^{1/2-1/q}$, the series $\sum c_n\|\phi_n\|_{p}$ is also convergent. Moreover, 
\[|c_{n_k}\lambda_{n_k}|(1-|\lambda_{n_k}|^2)^{-1/2}=\alpha_{n_k}|\lambda_{n_k}|(1-|\lambda_{n_k}|^2)^{1/q-1}\ge k,\qquad k\ge1.
\]
Thus we deduce that $\limsup_{N\to\infty}|\lambda_N c_N|(1-|\lambda_N|^2)^{-1/2}=\infty$. So \eqref{eq:sdqssqd3EZESDSXVF00SDN} is satisfied and $(B_n)_{n\geq 0}$ is not a Schauder basis for $H^p$, which concludes the proof of Theorem \ref{Th:CaracHp}.
\end{proof}

\section{The case of the disk algebra} \label{S:Shaprness}
One might ask whether Corollary~\ref{Cor:hp-blaschke} can be extended to the whole Banach space $X$. The previous section answers this question in the case where $X=H^p$. In this section, we consider the case of the disk algebra. For the disk algebra, the situation is very different from what we obtained for Hardy spaces and also much more complicated. For example, for constant sequences $\lambda_n \equiv \lambda$, the sequence of Blaschke products $(B_n)_{n\geq 0}$ is not  a Schauder basis in $\mathcal{A}(\mathbb{D})$, as explained in the following remark.

\begin{rem}\label{rem-A(D)-suiteCte}
Recall that the sequence of monomials $(z^n)_{n\geq 0}$ is a Schauder basis for $H^p$, $1<p<\infty$ (see, for example, \cite{Livre-TaylorCoef}*{p.108}), but it is not a Schauder basis for $\mathcal{A}(\mathbb{D})$ (by \cite{Zygmund-livre}*{Th. VIII.1.14}). Recall also that $C_{b_\lambda}$, the composition operator with symbol $b_\lambda$, is continuous from $H^p$ to itself, $1<p<\infty$ (by the Littlewood principle), and from $\mathcal{A}(\mathbb{D})$ to itself as well, and satisfies $C_{b_\lambda}^2=I$. Thus, it follows that $(b_{\lambda}^n)_{n\geq 0}$ is a Schauder basis for $H^p$, $1<p<\infty$,  but is not a Schauder basis for $\mathcal{A}(\mathbb{D})$.
\end{rem}



Remark \ref{rem-A(D)-suiteCte} and Theorem \ref{prop:completude-disc-algebra} suggest considering the case where $(\lambda_n)_{n\geq 0}$ is a non-Blaschke sequence satisfying $\sup_{n\geq 1}|\lambda_n|=1$. The next result shows that this is not sufficient to ensure that the sequence $(B_n)_{n\geq 0}$ forms a Schauder basis for $\mathcal{A}(\mathbb{D})$. 

\begin{prop}\label{Prop:A(D)}
There exists a non-Blaschke sequence $(\lambda_n)_{n\geq 1}$ satisfying $\sup_{n\geq 1}|\lambda_n|=1$, such that 
the sequence of finite Blaschke products $(B_n)_{n\geq 0}$ does not form a Schauder basis for $\mathcal{A}(\mathbb{D})$, with $B_n,\,n\ge0$, as in \eqref{D:def-Bn}.
\end{prop}
Recall that by Proposition \ref{Prop:unicity}, if  $f=\sum a_n B_n$ where the series converges in $\mathrm{Hol}(\mathbb{D})$, then the coefficients $a_n$ are given by \eqref{Eq:coefs}. In particular, to prove Proposition \ref{Prop:A(D)}, it is sufficient to prove the existence of a non-Blaschke sequence $(\lambda_n)_{n\geq 1}$ and a function $f\in \mathcal{A}(\mathbb{D})$ such that $\sup_{N\geq 1}\|R_Nf\|_\infty=\infty$, where
\[
R_N f=f-\sum_{n=0}^{N-1}a_nB_n=\left(T_{\overline{B}_N}f-\overline{\lambda}_NT_{\overline{B}_{N-1}}f(\lambda_N)\right)B_N.
\]
Let $\widetilde R_Nf=T_{\overline{B}_N}f-\overline{\lambda}_NT_{\overline{B}_{N-1}}f(\lambda_N)$. Then $\|R_Nf\|_\infty=\|\widetilde R_Nf\|_\infty$.
 So to prove the proposition, it is sufficient to prove the existence of a sequence $(\lambda_n)_{n\geq 1}$ and a function $f\in \mathcal{A}(\mathbb{D})$ such that $\sup_{N\geq 1} \|\widetilde R_Nf\|_\infty=\infty$. In particular, Proposition~\ref{Prop:A(D)} follows immediately from the next result. 
\begin{lem}\label{Lemma:Key_A(D)}
There exists a non-Blaschke sequence $(\lambda_n)_{n\geq 1}$ satisfying $\sup_{n\geq 1}|\lambda_n|=1$ and a function $f\in \mathcal{A}(\mathbb{D})$ such that
\[\sup_{N\geq 1}\|T_{\overline{B}_N}f-\overline{\lambda}_NT_{\overline{B}_{N-1}}f(\lambda_N)\|_\infty=\infty,
\]
where $B_n,\,n\ge0$, are given by \eqref{D:def-Bn}.
\end{lem}
\begin{proof}Let $h\in \mathcal{C}(\mathbb{T})$ be such that $\|h\|_\infty\le1$, but $Ph\in VMOA\setminus H^\infty$. Recall that the convex envelope of  the quotients of finite Blaschke products is dense in the unit ball of $\mathcal{C}(\mathbb{T})$, i.e.,
\[\overline{\mathrm{conv}\left(\frac{FBP}{FBP}\right)}=\{g\in \mathcal{C}(\mathbb{T})\,;\,\|g\|_\infty\le 1\},\]
where $FBP$ denotes the set of all finite Blaschke products. See \cite{MR3586184} for details. Then there exist rational functions $r_n$ with poles outside the closed unit disk, satisfying $\|r_n\|_\infty\le1$, and finite Blaschke products $\widetilde B_n$ such that 
\begin{equation}\label{Eq:Proof5.2-1}
    \left\|h-\frac{r_n}{\widetilde B_n}\right\|_\infty\longrightarrow0~\text{when}~n\to\infty.
\end{equation}
By  multiplying if necessary $r_n$ and $\widetilde B_n$ by some Blaschke factor, we can assume that the zero sets satisfy $Z(\widetilde B_n)\subset Z(\widetilde B_{n+1})$ and $1-1/n\in Z(\widetilde B_n)$. Let $(\lambda_n)_{n\geq 1}$ be the non-Blaschke sequence defined by the union of these zeros, arranged such that $(\widetilde B_n)$ is a subsequence of the Blaschke products $B_n$, $n\ge1$, given by  \eqref{D:def-Bn}. 

From \eqref{Eq:Proof5.2-1} and the continuity of the Riesz projection $P$ from $\mathcal{C}(\mathbb{T})$ to $VMOA$, we have
\[
\|Ph-T_{\overline{\widetilde B}_n}\,r_n\|_{VMO}\longrightarrow0~\qquad\text{when}~n\to\infty.
\]
In particular, for all $z\in\mathbb{D}$, $Ph(z)=\lim\limits_{n\to\infty}T_{\overline{\widetilde B}_n}\,r_n(z)$. But since $Ph\notin H^\infty$, we get $\sup_{z,w\in\mathbb{D}}|Ph(z)-Ph(w)|=\infty$. Thus, we deduce that 
\begin{equation}\label{Eq:lemA(A)1}
    \sup_{n\ge1,w,z\in\mathbb{D}}|T_{\overline{\widetilde B}_n}r_n(z)-T_{\overline{\widetilde B}_n}r_n(w)|=\infty.
\end{equation} 
Let $S_{n,w,z}$ be the linear functional on $\mathcal{A}(\mathbb{D})$ defined by \[S_{n,w,z}f=T_{\overline{\widetilde  B}_n}f(z)-T_{\overline{\widetilde  B}_n}f(w),~f\in \mathcal{A}(\mathbb{D}).\] Since $\|r_n\|_\infty\le1$, then by \eqref{Eq:lemA(A)1}, $\sup_{n,z,w}\|S_{n,w,z}\|=\infty$.

Therefore, by the Banach-Steinhaus theorem, there exists $f\in \mathcal{A}(\mathbb{D})$ such that
\[
\sup_{n,z,w}|T_{\overline{\widetilde B}_n}f(z)-T_{\overline{\widetilde  B}_n}f(w)|=\sup_{n,z,w}|S_{n,w,z}f|=\infty.
\]
 Since $(\widetilde B_n)$ is a subsequence of $(B_n)$ given by \eqref{D:def-Bn}, it follows that
\[\sup_{n,z,w}|T_{\overline{B}_n}f(z)-T_{\overline{B}_n}f(w)|\ge\sup_{n,z,w}|T_{\overline{\widetilde B}_n}f(z)-T_{\overline{\widetilde B}_n}f(w)|=\infty. \]

Now with $\widetilde R_nf=T_{\overline{B}_N}f-\overline{\lambda}_nT_{\overline{B}_{n-1}}f(\lambda_n)$, we have \[\widetilde R_nf(z)-\widetilde R_nf(w)=T_{\overline{B}_n}f(z)-T_{\overline{B}_n}f(w)
\] and so
\[\sup_n\|R_nf\|_\infty= \sup_n\|\widetilde R_nf\|_\infty\ge\frac12\sup_{n,z,w}|\widetilde R_nf(z)-\widetilde R_nf(w)|=\infty.\qedhere\]
\end{proof}

The construction of the sequence $(\lambda_n)_{n\geq 1}$ in the proof of Lemma \ref{Lemma:Key_A(D)} leaves some freedom, in particular the possibility of adding terms in the sequence. In other words, for every sequence of points in $\mathbb D$ having $(\lambda_n)_{n\geq 1}$ as a subsequence, the associated family of finite Blaschke products constructed by \eqref{D:def-Bn} does not form a Schauder basis for $\mathcal{A}(\mathbb D)$. Nevertheless, this does not solve the question for all non-Blaschke sequences $(\lambda_n)_{n\geq 1}$ satisfying $\sup_n|\lambda_n| = 1$. In particular, this leaves unknown the answer to the following intriguing  question.
\begin{question}
Does there exist a Shauder basis of $\mathcal{A}(\mathbb{D})$ of the form $(B_n)_{n\geq 1}$ where $(\lambda_n)_{n\geq 1}$ is a 
non-Blaschke sequence in $\mathbb{D}$ ? 
\end{question}

\section*{Acknowledgment}
 This work is supported by the Canada Research Chairs program and the Discovery Grant of NSERC. This work is also partially supported by the COMOP project of the French National Research Agency (grant ANR-24-CE40-0892-01).  The authors acknowledge the support of the CDP C2EMPI, as well as the French State under the France-2030 programme, the University of Lille, the Initiative of Excellence of the University of Lille, the European Metropolis of Lille for their funding and support of the R-CDP-24-004-C2EMPI project.

\bibliographystyle{plain}
\bibliography{Basis-References}
\end{document}